\documentclass{amsart}[12pt]
\usepackage[dvips]{graphicx}
\usepackage{cases}
\usepackage{lineno}

\makeatletter
\newtheorem*{rep@theorem}{\rep@title}
\newcommand{\newreptheorem}[2]{%
\newenvironment{rep#1}[1]{%
 \def\rep@title{#2 \ref{##1}}%
 \begin{rep@theorem}}%
 {\end{rep@theorem}}}
\makeatother

\newtheorem{prop}{\bf Proposition}[section]
\newtheorem{thm}[prop]{\bf Theorem}
\newreptheorem{theorem}{Theorem}
\newtheorem{lem}[prop]{\bf Lemma}

\newtheorem{claim}[prop]{\bf Claim}
\theoremstyle{definition}
\newtheorem{defi}[prop]{Definition}
\newtheorem{rmk}[prop]{Remark}

\makeatletter
\@addtoreset{equation}{section}

\makeatother

\newcommand{\map}{\rightarrow}

\newcommand{\Zk}{\mathbb{Z}/k\mathbb{Z}}

\title[Number of commensurable fibrations]{On the number of commensurable fibrations on a hyperbolic 3-manifold}
\author{Hidetoshi Masai}
\address{Department of Mathematical and Computing Sciences, Tokyo Institute of
Technology, O-okayama, Meguro-ku, Tokyo 152-8552 Japan}
\email{masai9 at is.titech.ac.jp}
\date{}

\begin{document}


\maketitle
\begin{abstract}
By a work of Thurston, it is known that if a hyperbolic fibred $3$-manifold $M$ has Betti number greater than 1, then 
$M$ admits infinitely many distinct fibrations.
For any fibration $\omega$ on a hyperbolic $3$-manifold $M$, 
the number of fibrations on $M$ that are commensurable in the sense of Calegari-Sun-Wang
 to $\omega$ is known to be finite.
In this paper, we prove that the number can be arbitrarily large.
\end{abstract}
\section{Introduction}
For a given fibration on a $3$-manifold,
there is an associated pair $(F,\phi)$ of the fibre surface $F$ and the monodromy $\phi:F\map F$.
Conversely, for any pair $(F,\phi)$ of a surface $F$  and an automorphism $\phi$ on $F$, we can construct
an associated $3$-manifold by taking the mapping torus.
Here, by an automorphism, we mean an isotopy class of self-homeomorphisms.
Calegari-Sun-Wang defined commensurability on pairs of type $(F,\phi)$.
\begin{defi}[\cite{CSW}]\label{def.cover}
A pair $(\widetilde{F},\widetilde{\phi})$ {\em covers} $(F,\phi)$ if there is a finite cover $\pi:\widetilde{F}\map F$ and representative homeomorphisms $\tilde{f}$ of $\widetilde\phi$ and $f$ of
$\phi$ so that $\pi \widetilde{f} = f \pi$ as maps $\widetilde{F}\map F$.
\end{defi}
\begin{defi}[\cite{CSW}]\label{def.cms}
Two pairs $(F_1, \phi_1)$ and $(F_2, \phi_2)$ are {\em commensurable} if there is a surface $\widetilde{F}$,
automorphisms $\widetilde{\phi_1}$ and $\widetilde{\phi_2}$, and nonzero integers $k_1$ and $k_2$, so that $(\widetilde{F},\widetilde\phi_i)$ covers
$(F_i,\phi_i)$ for $i = 1,2 $ and if $\widetilde{\phi_1}^{k_1} = \widetilde{\phi_2}^{k_2}$ as automorphisms of $\widetilde{F}$.
\end{defi}

Each commensurability class endows a partial order by the covering relation.
In \cite{CSW} (closed case) and \cite{Mas} (cusped case), it has been proved that in hyperbolic case, i.e. the case where
monodromies are pseudo-Anosov, 
there is a unique minimal (orbifold) pair in each commensurability class.
We consider fibrations on a connected orientable hyperbolic $3$-manifolds of finite volume $M$.
Two fibrations on $M$ are said to be {\em symmetric} if there exists a self-homeomorphism of $M$ that
maps one to the other.
If there are two non-symmetric but commensurable fibrations on $M$, 
then there are two distinct ways to cover the orbifold  $O$ associated to the minimal pair in the commensurability class.
Such coverings correspond to subgroups of the orbifold fundamental group of $O$ with the same index.
Since the number of subgroups of a given index in a finitely generated group is finite,
we see that on $M$, the number of fibrations that are commensurable to a fibration is always finite, see
\cite[Corollary 3.2]{CSW}.
Then it is natural to investigate how many commensurable fibrations can exist on $M$.
In \cite{Mas}, we observed if $M$ dose not have hidden symmetries, then $M$ has no non-symmetric but commensurable
fibrations.
We further showed that there are $3$-manifolds with hidden symmetries that have no non-symmetric but commensurable fibrations.
On the other hand, we also constructed manifolds that have two non-symmetric but commensurable fibrations.
In this paper we prove that the number of commensurable fibrations can be arbitrarily large.
\begin{thm}\label{thm.main}
For any $n\in\mathbb{N}$, there exists a hyperbolic $3$-manifolds with at least $n$ mutually 
non-symmetric but commensurable fibrations.
\end{thm}

\section{Proof of Theorem \ref{thm.main}}
In this section, we prove Theorem \ref{thm.main}.
We first prepare several lemmas.
Recall that for a given fibration on a $3$-manifold $M$, there is an associated element $\omega\in H^1(M;\mathbb{Z})$ where
for a loop $\gamma\subset M$, 
$\omega(\gamma)$ is equal to the intersection number of $\gamma$ and a chosen fibre surface in $M$.
\begin{lem}\label{lem.key}
For a given finite group $G$, there is a hyperbolic rational homology $3$-sphere $M$ and 
a hyperbolic fibred link $L\subset M$ such that
$G$ acts freely on $M$, 
and further $G$ acts on the components of $L$ freely and transitively i.e. 
the components of $L$ admit labels by elements of $G$ satisfying
\begin{eqnarray}
\displaystyle L = \bigsqcup_{g\in G}L_g,~\mathrm{and}~ \forall g,h\in G,~ h\cdot L_g = L_{hg}. \label{eq.label}
\end{eqnarray}
Furthermore, let $m_g$ denote the meridian of $L_g$ for each $g\in G$.
Then the sum of the duals $\sum_{g\in G} (m_g)^* \in H^1(M\setminus L;\mathbb{Z})$  
corresponds to a fibration on $M\setminus L$.
\end{lem}
\begin{rmk}
Since $M$ is a rational homology $3$-sphere, the set of meridians $\{m_g\}_{g\in G}$ is a basis of 
the free part of $H_1(M\setminus L;\mathbb{Z})$.
By this basis, we can uniquely determine the dual $(m_g)^*\in H^1(M\setminus L; \mathbb{Z})$ of $m_g\in H_1(M\setminus L;\mathbb{Z})$.
The set of duals $(m_g)^*_{g\in G}$ is a basis of $H^1(M\setminus L; \mathbb{Z})$.
\end{rmk}
The key ingredients of the proof of Lemma \ref{lem.key} are following three theorems.
The first theorem is due to Harer. 
\begin{thm}[\cite{Har}]\label{thm.Har}
Let $M$ be a smooth, closed orientable $3$-manifold.
The collection of all elements of $\pi_1(M)$ which can be represented by fibred knots is exactly the commutator subgroup
$[\pi_1(M),\pi_1(M)]$.
\end{thm}
\begin{rmk}\label{rem.meridian}
In \cite{Har}, a knot $K\subset M$ is called a fibred knot if $M\setminus K$ is a fibre bundle over $S^1$ with projection 
map $\phi$ and the restriction of $\phi$ on the meridian of $K$ is surjective.
\end{rmk}
The second theorem is due to Soma. 
\begin{thm}[\cite{Som}, Theorem 1.]\label{thm.Som}
Let $M$ be a smooth, closed orientable, irreducible $3$-manifold.
Then every fibred knot in $M$ is homotopic to a hyperbolic fibred knot.
\end{thm}
\begin{rmk}
In \cite{Som}, Soma has proved more strongly that every fibred knot in $M$ is fibre-concordant to a hyperbolic knot.
However to prove Lemma \ref{lem.key}, we do not need this fact.
\end{rmk}
Finally we will also refer to a theorem by Cooper-Long.
\begin{thm}[\cite{CL}, Theorem 2.6.]\label{thm.CL}
Let $G$ be a finite group. Then there is a hyperbolic rational homology $3$-sphere on which $G$
acts freely.
\end{thm}
\begin{proof}[Proof of Lemma \ref{lem.key}.]
By Theorem \ref{thm.CL},  there is a hyperbolic rational homology $3$-sphere $M$ on which $G$ acts freely.
We consider the quotient $N:=M/G$.
Let $p:M\map N$ denote the projection map.
The map $p$ can also be regarded as a covering map associated to a surjection $\rho:\pi_1(N)\map G$.
By Theorem \ref{thm.Har}, there is a fibred knot $K'$ such that $[K']\in[\pi_1(M), \pi_1(M)]$.
Since $G$ is a finite group, by taking power of $[K']$ if necessary, 
we see that there is a fibred knot $K\subset N$ such that $\rho([K]) = 1_G$,
where $1_G$ is the unit element of $G$.
Furthermore, by Theorem \ref{thm.Som}, we may assume that $K$ is a hyperbolic fibred knot in $N$.
By appealing to standard facts of covering spaces, we see that $L:=p^{-1}(K)\subset M$ has exactly $|G|$ components and 
admits labelling on its components that satisfies (\ref{eq.label}).
Note that since $M$ is a rational homology $3$-sphere, the quotient $N$ also satisfies $H^1(N;\mathbb{Q}) = 0$.
By Remark \ref{rem.meridian},
the dual $m^*\in H^1(N\setminus K;\mathbb{Z})$ of the meridian $m$ of $K$ corresponds to the fibration.
By construction, each component of $p^{-1}(m)$ is the meridian of some $L_g$ and their sum in the first
cohomology $\sum_{g\in G}(m_g)^*\in H_1(M)$ corresponds to 
the fibration which is the lift of the fibration associated to $m^*$ on $N\setminus K$.
\end{proof}

The following lemma follows from standard facts of covering spaces and fibrations.
Note that for a given fibration on $M$ with fibre $F$, there is an associated exact sequence 
\[1\map \pi_1(F)\stackrel{\iota}{\map} \pi_1(M)\rightarrow \pi_1(S^1)\map 1.\]

\begin{lem}\label{lem.cover}
Let $G$ be a finite group and $M$ a fibred 3-manifold with fibre surface $F$.
Further, let $p:M'\map M$ be a covering corresponding to a 
surjection $\rho:\pi_1(M)\map G$.
Then, $p^{-1}(F)\subset M'$ has 
\[\frac{|G|}{|\rho(\iota(\pi_1(F)))|}\]
components and each component is a degree $|\rho(\iota(\pi_1(F)))|$ covering of $F$.
\end{lem}

The following lemma follows from the theory of fibred cones due to Thurston \cite{Thu.norm}.
\begin{lem}\label{lem.cone}
Let $M$ be a hyperbolic $3$-manifold of the first Betti number $>1$.
Suppose $\omega\in H^1(M;\mathbb{Z})$ corresponds to a fibration on $M$.
Then for any $\omega'\in H^1(M;\mathbb{Z})$, there exists $N\in\mathbb{N}$ such that for any $n>N$,
$n\omega + \omega'$ corresponds to a fibration on $M$.
\end{lem}
\begin{proof}
Thurston introduced a semi-norm, so called the Thurston norm, on $H^1(M;\mathbb{R})$.
In \cite{Thu.norm}, Thurston proved that the unit ball of the Thurston norm is a compact convex polygon.
Furthermore, Thurston showed the following.
Let $C_{\Delta}$ be a cone over a top dimensional face $\Delta$ of the unit ball.
Then if an integral element 
(i.e. an element on $H^1(M:\mathbb{Z}) \subset H^1(M;\mathbb{R})$)  
of the interior of $C_{\Delta}$ corresponds to a fibration on $M$, 
then every integral element in the interior of $C_{\Delta}$ corresponds to a fibration 
(see \cite{Thu.norm}, or \cite{Kapo} for details).
Such a cone is called a {\em fibered cone}.
Since the sequence of lines that connects the origin and $n\omega + \omega'$ converges to the line that connects
the origin and $\omega$ as $n$ goes to infinity, we see that for large enough $n$, $n\omega + \omega'$ is in 
the interior of the same cone as $\omega$.
Thus we see that $n\omega + \omega'$ corresponds to a fibration.
\end{proof}
\begin{proof}[Proof of Theorem \ref{thm.main}.]
We first apply Lemma \ref{lem.key} to $G = \Zk$.
We use the notations in Lemma \ref{lem.key}.
Let $\omega = \sum_{[s]\in\Zk } (m_{[s]})^*$ and $\omega' = \sum_{[s]\in \Zk} \mathrm{mod}_k(s)(m_{[s]})^* $ where
$\mathrm{mod}_k:\mathbb{Z}\map \{1,2,...,k\}$ denote the quotient modulo $k$, here we assign $k$ instead of $0$ to integers divisible by $k$.
Then by Lemma \ref{lem.cone}, we see that for large enough $l$, $\eta := l\omega + \omega'$ corresponds to a fibration on
$M\setminus L$.
Let $F_{\eta}$ be the fibre surface of the fibration associated to $\eta$.
Then by appealing to the fibration structure, we have an exact sequence
\begin{eqnarray}
\label{fibre.exact} 1\map \pi_1(F_{\eta})\stackrel{\iota}{\map} \pi_1(M\setminus L)\stackrel{\eta}{\rightarrow} \pi_1(S^1)\map 1.
\end{eqnarray}
The map $\eta$ factors through the abelianization $\mathrm{ab} : \pi_1(M\setminus L) \map H_1(M\setminus L;\mathbb{Z})$ since $\pi_1(S^1)\cong \mathbb{Z}$ is abelian.
This implies that $A_{\eta} := \mathrm{ab}(\iota(\pi_1(F_{\eta})))$ is equal to 
$\ker(\eta)/\mathrm{Tor}\subset H_1(M\setminus L;\mathbb{Z})$ which is isomorphic to $\mathbb{Z}^{k-1}$.
Note that the number $k$ of components of $L$ is the first Betti number of $M\setminus L$ since
$M$ is a rational homology 3-sphere.
We note that the set of the dual of meridians $\{(m_{[s]})^*\}_{[s]\in\Zk}$ is a basis of $H^1(M\setminus L; \mathbb{Z})$.
In the first homology, $\{(m_{[s]})_*\}_{[s]\in\Zk}$ is a basis of the free part of $H_1(M\setminus L;\mathbb{Z})$.
The action of $\Zk$ on $M$ induces an action on $H_1(M\setminus L;\mathbb{Z})$ and since $[s']\cdot L_{[s]} = L_{[s+s']}$,
we see that $[s']\cdot m_{[s]} = m_{[s+s']}$.
The free part of the kernel of $[s]\cdot \eta$ in $H_1(M\setminus L;\mathbb{Z})$ is equal to $[s]\cdot A_\eta$.
Then we consider the dynamical covering $p_d:M_d\map M\setminus L$ of degree $d$, that is, the covering corresponding to taking the $d$-th power of the monodromy map of the fibration associated to $\eta$.
This covering is a covering associated to the surjection map
\[\pi_1(M)\stackrel{\rm ab}{\rightarrow}H_1(M)\stackrel{\eta}{\rightarrow}\mathbb{Z}\map \mathbb{Z}/(d\mathbb{Z}).\]
We will compute the image $\eta([s]\cdot A_{\eta})\subset \mathbb{Z}$.
\begin{claim}\label{claim}
$\eta([s]\cdot A_{\eta}) = \gcd(s,k-s)\cdot\mathbb{Z}$ where $\gcd$ is the greatest common devisor.
\end{claim}
\begin{proof}[Proof of Claim \ref{claim}.]
Note that since torsion part of the first homology maps to the zero by any element of the first cohomology,
to see the image we only need to consider the free part.
Let $\alpha = \sum_{[i]\in\Zk} \alpha_i(m_{[i]})_{*}$ be an element of $\eta([s]\cdot A_{\eta})$.
Since $[s]\cdot A_{\eta} \subset \ker([s]\cdot \eta)$, we have
\begin{eqnarray} 
([s]\cdot\eta)(\alpha) = \sum_{[i]\in\Zk} (l+\mathrm{mod}_k(i+s))\alpha_i = 0. \label{eq.def}
\end{eqnarray} 
Then,
\begin{eqnarray}
\eta(\alpha) & = & \sum_{[i]\in\Zk} (l+\mathrm{mod}_k(i))\alpha_i \\
\label {eq.second} & = & (-s)(\alpha_1+ \cdots + \alpha_{k-s}) + (k-s) (\alpha_{k-s+1} + \cdots + \alpha_k)
\end{eqnarray}
Here to get (\ref{eq.second}), we subtracted (\ref{eq.def}).
Thus we have
$\eta([s]\cdot A_{\eta})\subset \gcd(s,k-s)\cdot\mathbb{Z}.$
Since $\displaystyle\eta = \sum_{[s]\in \Zk} (l+\mathrm{mod}_k(s))(m_{[s]})^* $, 
we see that for $j\not= k-s$
\[\beta := (l+\mathrm{mod}_k(j+s+1))(m_{[j]})_* - (l+\mathrm{mod}_k(j+s))(m_{[j+1]})_*\in \ker([s]\cdot\eta).\]
Then, for any $a\in\mathbb{Z}$
\begin{equation}
  \eta(a\beta)=\begin{cases}
    -sa ,   ~\text{if $1\leq j < k-s$,}\\
    (k-s)a,  ~\text{if $k-s<j<k-1$.}
  \end{cases}
\end{equation}
Hence by Euclidean algorithm, we see that $\eta([s]\cdot A_{\eta})= \gcd(s,k-s)\cdot\mathbb{Z}.$
This completes the proof of the claim.

\end{proof}

By Lemma \ref{lem.cover}, if $\gcd(s,m-s)|d$ then the pre-image $p_d^{-1}([s]\cdot F_{\eta})$ of the fibre has
$d/\gcd(s,m-s)$ components.
In other words, if we restrict $p_d$ to a component of $p_d^{-1}([s]\cdot F_{\eta})$, then it is of degree $\gcd(s,m-s)$.
Then suppose $k = n!$ and $d = n!$, then for each $1 \leq i \leq n$, the lift of the fibration corresponding to $[i]\cdot \eta$
by $p_d$ has fibre with Euler characteristics $i$ times that of $F_{\eta}$.
In particular, they are not symmetric to each other, however by construction, they are mutually commensurable.
This completes the proof of Theorem \ref{thm.main}.
\end{proof}

\begin{rmk}
Since $\Zk$ is a finite group, by a similar argument of the proof of Lemma \ref{lem.cone},
we may further assume that $[i]\cdot\eta$'s are on the same fibred cone for all $[i]\in\Zk$.
Hence the fibrations we constructed in Theorem \ref{thm.main} can be on the same fibred cone.
In this case, by a theorem of Fried \cite[Theorem 7]{Fri}, we see that the suspension flow by those fibrations are isotopic.
\end{rmk}

\section*{acknowledgements}
I would like to thank Sadayoshi Kojima for helpful discussion and encouragement.
This work was partially supported by JSPS Research Fellowship for Young Scientists.


\end{document}